\documentclass[12pt]{article}
\usepackage{e-jc}
\usepackage{amsthm,amsmath,amssymb}

\usepackage[colorlinks=true,citecolor=black,linkcolor=black,urlcolor=blue]{hyperref}

\usepackage{graphicx}

\numberwithin{equation}{section}
\numberwithin{figure}{section}
\theoremstyle{plain}
\newtheorem{thm}{Theorem}[section]
\newtheorem{lem}[thm]{Lemma}
\newtheorem{cor}[thm]{Corollary}

\newtheorem{conj}[thm]{Conjecture}

\theoremstyle{remark}

\newcommand{\M}{\operatorname{M}}

\title{Proof of a refinement of Blum's conjecture on hexagonal dungeons}
\author{Tri Lai\footnote{This research was supported in part by the Institute for Mathematics and its Applications with funds provided by the National Science Foundation.}\\
\small Institute for Mathematics and its Applications\\[-0.8ex]
\small University of Minnesota\\[-0.8ex]
\small Minneapolis, MN 55455\\
\small\tt tmlai@ima.umn.edu
}
\date{\small Mathematics Subject Classifications: 05A15, 05C70}

\begin{document}
\maketitle

\begin{abstract}
Matt Blum conjectured that the number of  tilings of a hexagonal dungeon with side-lengths $a,2a,b,a,2a,b$ (for $b\geq2a$) equals $13^{2a^2}14^{\lfloor a^2/2\rfloor}$. Ciucu and the author of the present paper proved the conjecture by using Kuo's graphical condensation method. In this paper, we investigate  a 3-parameter refinement of the conjecture and its application to enumeration of tilings of several new types of the hexagonal dungeons.


\end{abstract}

\section{Introduction}
A lattice partitions the plane into fundamental regions. A \textit{(lattice) region} considered in this paper is a finite connected union of fundamental regions. We call the union of any two fundamental regions sharing an edge a \textit{tile}. We would like to know how many different ways to cover a region by tiles so that there are no gaps or overlaps; and such coverings are called \textit{tilings}. We use the notation $\M(R)$ for the number of tilings of a region $R$, and $\mathcal{M}(R)$ for the set of all tilings of $R$.

Consider the lattice obtained from the triangular lattice by drawing in all attitudes of each unit triangle. The resulting lattice is usually called \text{$G_2$-lattice}, since it is the lattice corresponding to the affine Coxeter group $G_2$. On the $G_2$-lattice, Blum investigated a variation of Aztec dungeon (see \cite{Ciucu}) called \textit{hexagonal dungeon}. In particular, we draw a hexagonal contour of side-lengths\footnote{The unit here is the side-length of the unit triangles.} $a,2a,b,a,2a,b$ (in cyclic order, start by the west side) as the light bold contour in Figure \ref{hexagon}. We draw next a jagged boundary running along the hexagonal contour (see the dark bold closed path in Figure \ref{hexagon}), and denote by $HD_{a,2a,b}$ the region restricted by the boundary. Blum found a striking pattern of the numbers of tilings of the hexagonal dungeons, which led him to his well-known conjecture that the hexagonal dungeon $HD_{a,2a,b}$ has $13^{2a^2}14^{\lfloor\frac{a^2}{2}\rfloor}$ tilings, when $b\geq 2a$. Fourteen years latter, Ciucu and the author proved the conjecture in \cite{CL} by using Kuo's graphical condensation method \cite{Kuo}. However, the proof did not explain the (surprising) appearance of the numbers 13 and 14 in the Blum's formula. In this paper, we consider generating function of the tilings of the hexagonal dungeons and give an explanation for the appearance of  the numbers $13$ and $14$.

\begin{figure}\centering
\includegraphics[width=12cm]{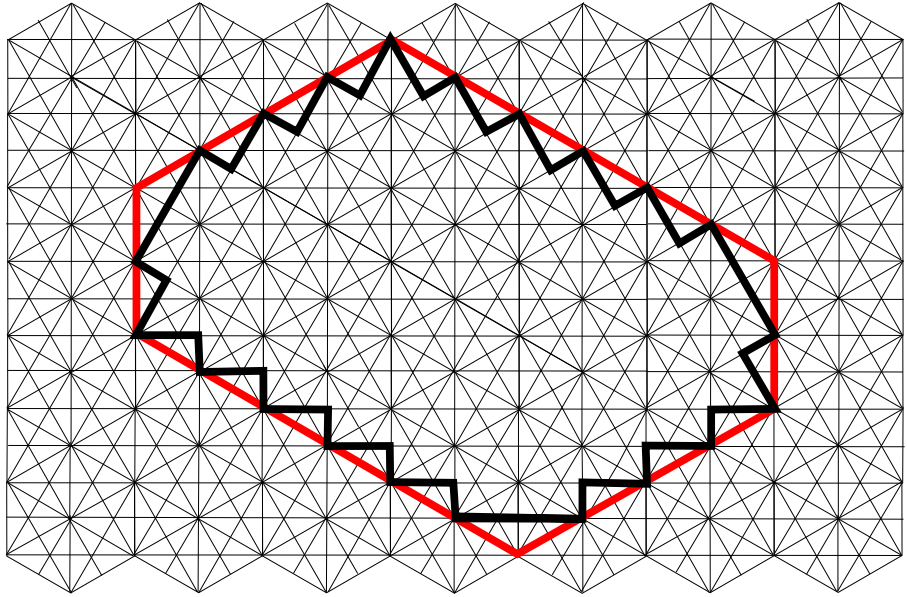}
\caption{The hexagonal dungeon of sides $2,$ $4,$ $6,$ $2,$ $4,$ $6$ (in cyclic order, starting from the western side). This Figure first appeared in
\cite{CL}.}
\label{hexagon}
\end{figure}

 \begin{figure}\centering
\includegraphics[width=10cm]{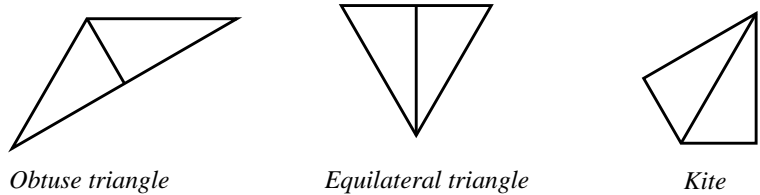}
\caption{Three types of tiles of a hexagonal dungeon.}
\label{dungeontiletype}
\end{figure}

The tiles in a hexagonal dungeon have three possible shapes: an obtuse triangle, an equilateral triangle,  and a kite (see Figure \ref{dungeontiletype}). We consider the following generating functions
\begin{equation}\label{define}
F(x,y,z)=\sum_{T\in\mathcal{M}(HD_{a,2a,b})}x^{m}y^{n}z^{l},
\end{equation}
 where $m$, $n$, $l$ are respectively the numbers of obtuse triangle tiles, equilateral triangle tiles, and kite tiles in the tiling $T$ of $HD_{a,2a,b}$. We call $F(x,y,z)$ the \textit{tiling generating function} of the hexagonal dungeon. Our goal is to prove the following refinement of Blum's conjecture.

\begin{thm}[Weighted Hexagonal Dungeon Theorem]\label{weighthd}
Assume $a$ and $b$ are two positive integers so that $b\geq 2a$.
Then the tiling generating function of the hexagonal dungeon $HD_{a,2a,b}$ is given by
\begin{align}
F(x,y,z)=&\gamma(a)x^{3ab-2a^2+3\lfloor\frac{a^2}{2}\rfloor} y^{6ab-a^2+3a-b-2\lfloor\frac{a^2}{2}\rfloor}z^{9ab+3a^2-7\lfloor\frac{a^2}{2}\rfloor-2b-12a}\notag\\
&\times (x^6+3x^4y^2+3x^2y^4+y^6+2x^3z^3+2xy^2z^3+z^6)^{2a^2}\notag\\
&\times\left((x^4+2x^2y^2+y^4+x^3z+xy^2z-y^2z^2+xz^3+z^4)(x^2+y^2-xz+z^2)\right)^{\lfloor\frac{a^2}{2}\rfloor},
\end{align}
where 
$\gamma(a)=(y/z)^{1-(-1)^a}$ (i.e. $\gamma(a)$ is $1$ if $a$ is even, and $y^2/z^2$ if $a$ is odd).
\end{thm}

This paper is organized as follows. In Section 2,  we recall the definition of two important regions $D_{a,b,c}$ and $E_{a,b,c}$, which were first introduced in \cite{CL}. In addition, we state a result on  their tiling generating functions (see Theorem \ref{newmain}), which is the key to prove Theorem \ref{weighthd}. Next, we prove Theorem \ref{newmain} in Section 3, and use this to prove Theorem \ref{weighthd} in Section 4. Section 5 uses Theorem \ref{weighthd} to enumerate of tilings of several new types of hexagonal dungeons.  Finally, Section 6 is devoted for an open problem on hexagonal dungeons with defects.

\section{The regions $D_{a,b,c}$ and $E_{a,b,c}$ and their tilings generating functions.}

First,  we recall briefly the definition of the two regions $D_{a,b,c}$ and $E_{a,b,c}$, which were first introduced in \cite{CL}.

\begin{figure}\centering
\includegraphics[width=10cm]{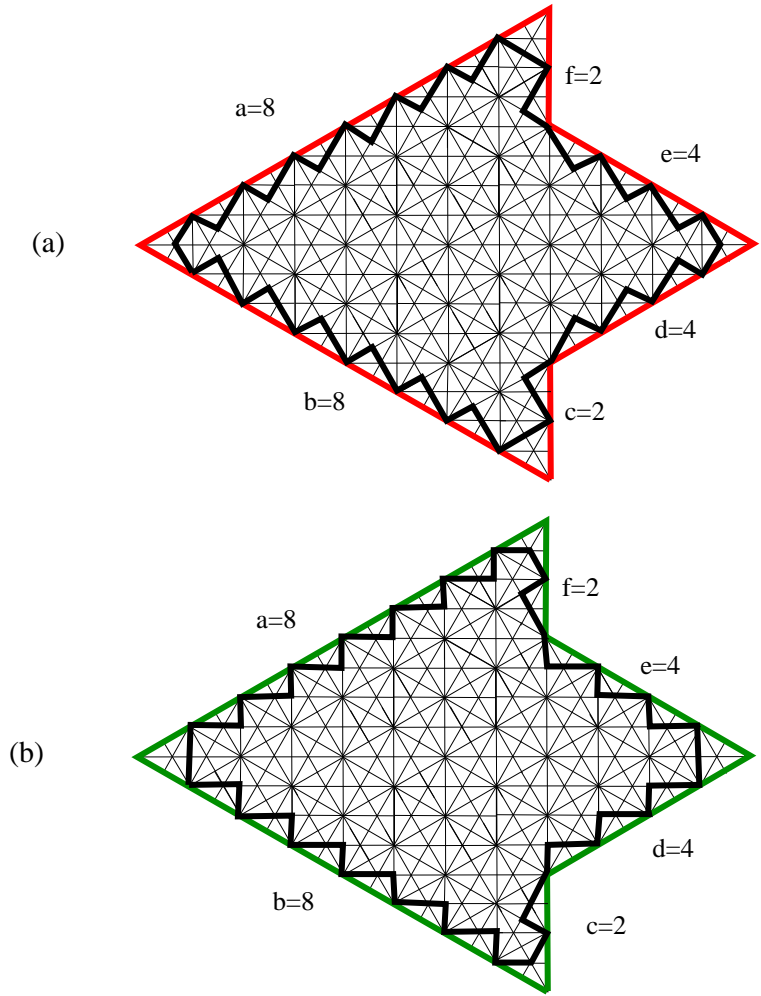}
\caption{Example of the case $a>c+d$: the regions $D_{8,8,2}$ (a) and $E_{8,8,2}$ (b). This Figure first appeared in
\cite{CL}.}
\label{DEcombine1}
\end{figure}

\begin{figure}\centering
 \includegraphics[width=10cm]{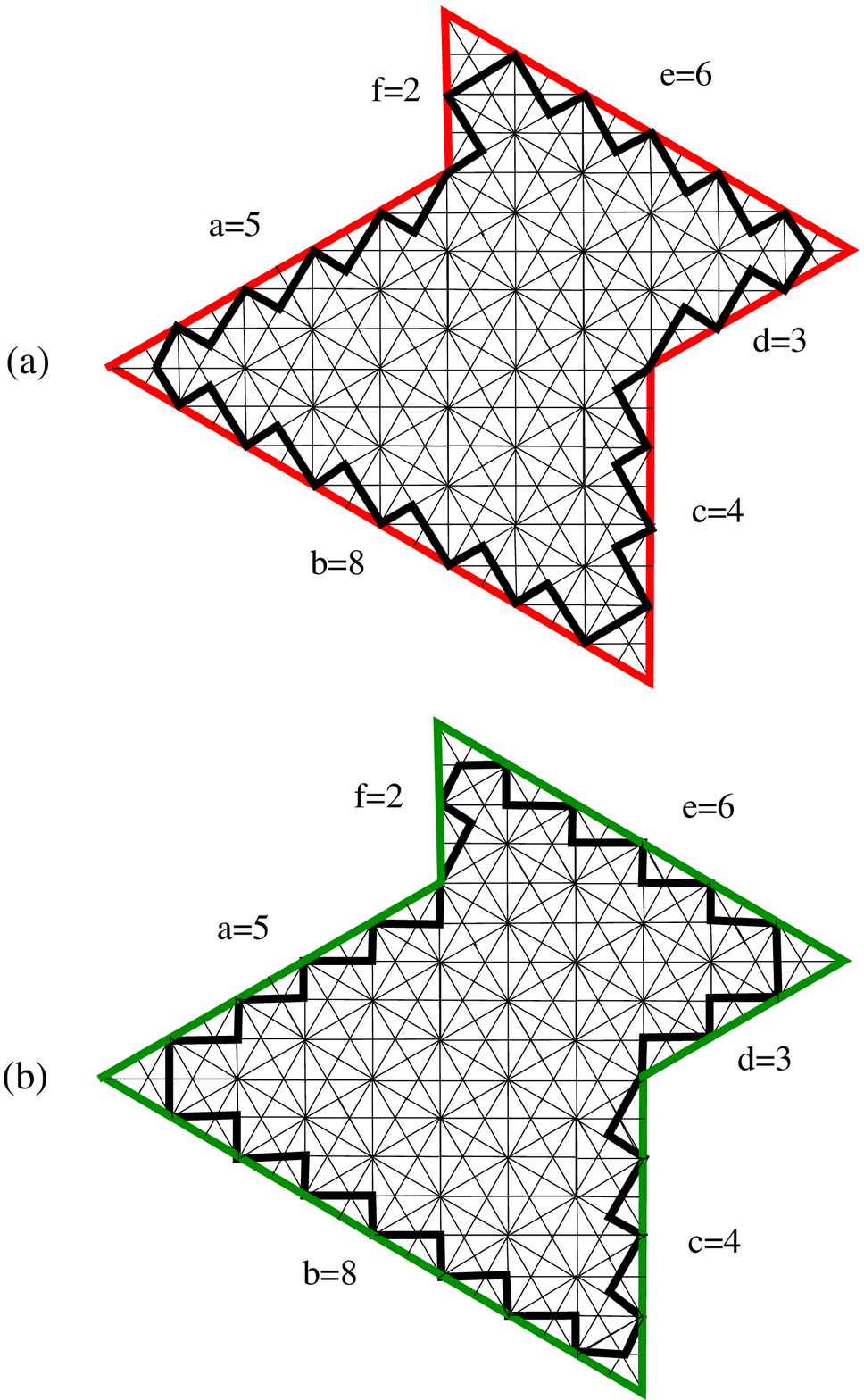}
\caption{Example of the case $a\leq c+d$: the regions $D_{5,8,4}$ (a) and $E_{5,8,4}$ (b). This Figure first appeared in
\cite{CL}.}
\label{DEcombine2}
\end{figure}

Let $a,b,c,d,e,f$ be six non-negative integers.  Starting from a vertex of some unit triangle on $G_2$-lattice, we travel along lattice lines $a$ units southwest, $b$ units southeast, $c$ units north, $d$ units northeast, and $e$ units northwest. We adjust $e$ so that the ending point is on the same vertical line as the starting point. Finally, we close the contour by go $f$ units north or south.

 The contour is illustrated in Figure \ref{DEcombine1}, for the case $a> c+d$; and  in  Figure \ref{DEcombine2},  for the case $a\leq c+d$. If  some side of the contour has length $0$, we assume that the side shrinks to a point. Denote  by $\mathcal{C}(a,b,c)$ the resulting contour.

Based on the contour $\mathcal{C}(a,b,c)$, we define two lattice regions $D_{a,b,c}$ and $E_{a,b,c}$  determined by the dark jagged boundaries as in Figure \ref{DEcombine1}, for the case $a>c+d$, and by Figure \ref{DEcombine2}, for the case $a\leq c+d$.

We already showed  in \cite{CL} that the closure of the contour $\mathcal{C}(a,b,c)$, the choice of $e$, and the existence of tilings in the regions $D_{a,b,c}$ and $E_{a,b,c}$ require  $b\geq 2$, $d=2b-a-2c$, $e=b+d-a=3b-2a-2c,$ and $f=|a-c-d|=|2a+c-2b|$.  Moreover, it has been shown in \cite{CL} that the numbers of tilings of $D_{a,b,c}$ and $E_{a,b,c}$ is given by powers of $13$ and $14$ (see Theorem 3.1).

\medskip

Next, we consider the tiling generating functions of $D_{a,b,c}$ and $E_{a,b,c}$ as follows. Define
\begin{equation}
\Phi(a,b,c):=\Phi(a,b,c)(x,y):=\sum_{T\in \mathcal{M}(D_{a,b,c})}x^{m}y^{n},
\end{equation}
 where $m$ and $n$ are respectively the numbers of obtuse triangle tiles and equilateral triangle tiles in tiling $T$ of $D_{a,b,c}$. Similarly, we set
\begin{equation}
\Psi(a,b,c):=\Psi(a,b,c)(x,y):=\sum_{T'\in \mathcal{M}(E_{a,b,c})}x^{k}y^{l},
\end{equation}
 where $k$ and $l$ are the numbers of obtuse triangle tiles and equilateral triangle tiles in tiling $T'$ of $E_{a,b,c}$.

\medskip

Define three new functions as follows:
\begin{equation}
g(a,b,c):=(b-a)(b-c)+\left\lfloor\frac{(a-c)^2}{3}\right\rfloor,
\end{equation}
\begin{equation}
q(a,b,c):=  \left\lfloor\frac{(a-b+c)^2}{4}\right\rfloor,
\end{equation}
and
\begin{equation}
\tau(a,b,c):=\begin{cases}
4a^2+8b^2+4c^2-10ab-10bc+6ac-2b+c & \text{if $2a+c-2b>0$;}\\
4a^2+8b^2+4c^2-10ab-10bc+6ac+6a-8b+4c & \text{otherwise.}
\end{cases}
\end{equation}

Denote by $P(x,y,z)$ and $Q(x,y,z)$ the two polynomials appearing in Theorem \ref{weighthd}, i.e.
\begin{equation}
P(x,y,z):=x^6+3x^4y^2+3x^2y^4+y^6+2x^3z^3+2xy^2z^3+z^6
\end{equation}
and
\begin{equation}
Q(x,y,z):=(x^4+2x^2y^2+y^4+x^3z+xy^2z-y^2z^2+xz^3+z^4)(x^2+y^2-xz+z^2).
\end{equation}
In addition, we define two simple functions $h(a,b,c)$ and $p(a,b,c)$ as
\begin{equation}
h(a,b,c)=
\begin{cases}
(x^4+2x^2y^2+y^4+x)x/y &\text{ if $3b+a-c\equiv 5\pmod{6};$}\\
(x^3+xy^2+1)x/y   &\text{ if $3b+a-c\equiv 1\pmod{6};$}\\
(x^2+y^2)&\text{ if $3b+a-c\equiv 4\pmod{6};$}\\
y&\text{ if $3b+a-c\equiv 3\pmod{6};$}\\
1 &\text{ otherwise,}
\end{cases}
\end{equation}
 and
 \begin{equation}
p(a,b,c)=
\begin{cases}
(x^4+2x^2y^2+y^4+x)x/y &\text{ if $3b+a-c\equiv 1\pmod{6};$}\\
(x^3+xy^2+1)x/y   &\text{ if $3b+a-c\equiv 5\pmod{6};$}\\
(x^2+y^2)&\text{ if $3b+a-c\equiv 2\pmod{6};$}\\
y&\text{ if $3b+a-c\equiv 3\pmod{6};$}\\
1 &\text{ otherwise.}
\end{cases}
\end{equation}

The generating functions $\Phi(a,b,c)$ and $\Psi(a,b,c)$ are given by the theorem stated below.
\begin{thm}\label{newmain}
Assume that $a$, $b$ and $c$ are three non-negative integers satisfying $b\geq 2$, $d:=2b-a-2c\geq0$ and $e:=3b-2a-2c\geq0$. Then
\begin{equation}\label{maineq1}
\Phi(a,b,c)=h(a,b,c)x^{2g(a,b,c)+3q(a,b,c)}y^{\tau(a,b,c)}P(x,y,1)^{g(a,b,c)}Q(x,y,1)^{q(a,b,c)}
\end{equation}
and
\begin{equation}\label{maineq2}
\Psi(a,b,c) =p(a,b,c)x^{2g(a,b,c)+3q(a,b,c)}y^{\tau(a,b,c)}P(x,y,1)^{g(a,b,c)}Q(x,y,1)^{q(a,b,c)}.
\end{equation}
\end{thm}
Theorem \ref{newmain} is the key in proving our main Theorem \ref{weighthd} in Section 4.

\medskip

We conclude this section by noticing that Ciucu proved implicitly the formulas of $\Phi(n,n,0)$ and $\Psi(n,n,0)$ (under the form of tiling generating functions of the Aztec dungeons in Theorem 3.1 and Proposition 3.6 of \cite{Ciucu}). In other word, our Theorem \ref{newmain} generalizes Ciucu's Theorem 3.1 and Proposition 3.6 in \cite{Ciucu}.

\section{Proof of Theorem \ref{newmain}}
Before proving Theorem \ref{newmain}, we need several definitions and terminology as follows.

A \textit{perfect matching} of a graph $G$ is a collection of disjoint edges covering all vertices of $G$. The \textit{dual graph} $G$ of $R$ is the graph whose vertices are fundamental regions in $R$ and whose edges connect precisely two fundamental regions sharing an edge.  The tilings of a region can be identified with the perfect matchings of it dual graph.   In the view of this, we use the notation $\M(G)$ for the number of perfect matchings of $G$.

In the weighted case, $\M(G)$ defines the sum of weights of perfect matchings of $G$, where the \emph{weight} of a perfect matching is the product of weights of all constituent edges. Define similarly the weighted sum $\M(R)$ of tilings of a weighted region $R$. Each edge of the dual graph $G$ of $R$ has the same weight as that of its corresponding tile in $R$.

Consider the following recurrences (\ref{WR1})--(\ref{WR5}), where the notations $\bigstar(a,b,c)$ and $\lozenge(a,b,c)$ have been used to indicate some polynomials in $\mathbb{Z}[x,y]$.

\begin{equation}\tag{R1}
\begin{split}
\bigstar(a,b,c)\bigstar(a-3,b-3,c-2)=xy^2\bigstar(a-2,b-1,c)\bigstar(a-1,b-2,c-2)\\+x^2y^4\bigstar(a-1,b-1,c-1)\bigstar(a-2,b-2,c-1).
\end{split} \label{WR1}
\end{equation}

\medskip

\begin{equation}\tag{R2}
\bigstar(a,b,c)\bigstar(a-2,b-2,c)=x^2y^4\bigstar(a-1,b-1,c)^2\\+\bigstar(a,b,c+1)\bigstar(a-2,b-2,c-1).
\label{WR2}
\end{equation}

\medskip

\begin{equation}\tag{R3}
      \bigstar(a,b,0)\bigstar(a-2,b-2,0)=x^2y^4\bigstar(a-1,b-1,0)^2\\+\bigstar(a,b,1)\bigstar(3b-2a,2b-a,1).
      \label{WR3}
\end{equation}

\medskip

\begin{equation}\tag{R4}
  \begin{split}
         \bigstar(a,b,c)\bigstar(a-2,b-3,c-2)=xy^2\bigstar(a-1,b-1,c)\bigstar(a-1,b-2,c-2)\\+xy^2\bigstar(a-2,b-2,c-1)\bigstar(a,b-1,c-1).
  \end{split}\label{WR4}
\end{equation}

\medskip

\begin{equation}\tag{R5}
\begin{cases}
    \begin{split}
       \bigstar(a,b,c)\bigstar(a-2,b-3,c-2)=xy^2\lozenge(c,b-1,a-1)\bigstar(a-1,b-2,c-2)\\+xy^2\bigstar(a-2,b-2,c-1)\bigstar(a,b-1,c-1);
    \end{split}\\

    \begin{split}
      \lozenge(a,b,c)\lozenge(a-2,b-3,c-2)= xy^2\bigstar(c,b-1,a-1)\lozenge(a-1,b-2,c-2)\\+xy^2\lozenge(a-2,b-2,c-1)\lozenge(a,b-1,c-1).
    \end{split}
    \end{cases}
    \label{WR5}
\end{equation}

We notice that the recurrence (\ref{WR1}) implies the recurrence in Lemmas 4.1 and 5.1 in \cite{CL} by specializing $x=y=1$. Similarly, (\ref{WR2}) and (\ref{WR3}) are weighted versions of the recurrence in Lemmas 4.2(a) and 5.2(a) and  the recurrence in Lemmas 4.2(b) and 5.2(b) in \cite{CL}, respectively. Finally, the $x=y=1$ specializations of (\ref{WR4}) and (\ref{WR5}) give the recurrences in Lemmas 4.3 and 5.3 of \cite{CL}.
\bigskip

Next, we show that $\Phi(a,b,c)$ and $\Psi(a,b,c)$ satisfy the above recurrences (\ref{WR1})--(\ref{WR5}) (with certain constraints) by using the following Kuo's Condensation Theorem.

\begin{thm}[Kuo's Condensation Theorem \cite{Kuo}]\label{kuothm}
Assume that $G$ is a planar bipartite graph, and that $V_1$ and $V_2$ are its vertex classes with $|V_1|=|V_2|$. Let $u,v,w,t$ be four vertices appearing in a cyclic order on a face of $G$. Assume in addition that $u,w\in V_1$ and $v,t\in V_2$. Then
\begin{equation}\label{kuoeq}
\M(G)\M(G-\{u,v,w,t\})=\M(G-\{u,v\})\M(G-\{w,t\})+\M(G-\{u,t\})\M(G-\{w,v\}).
\end{equation}
\end{thm}

\begin{lem}\label{con1} Let $a$, $b$ and $c$ be non-negative integers so that  $b\geq5$,  $c\geq 2$, $d:=2b-a-2c\geq0$, and $3b-2a-2c\geq0$. Assume in addition that $a\geq c+d+1$. Then $\Phi(a,b,c)$ and $\Psi(a,b,c)$ both satisfy the recurrence (\ref{WR1}), i.e. we have
\begin{equation} \label{eq1-1}
\begin{split}
\Phi(a,b,c)\Phi(a-3,b-3,c-2)=xy^2\Phi(a-2,b-1,c)\Phi(a-1,b-2,c-2)\\+x^2y^4\Phi(a-1,b-1,c-1)\Phi(a-2,b-2,c-1)
\end{split}\end{equation}
 and
\begin{equation}\label{eq2-1}
 \begin{split}
\Psi(a,b,c)\Psi(a-3,b-3,c-2)=xy^2\Psi(a-2,b-1,c)\Psi(a-1,b-2,c-2)\\+x^2y^4\Psi(a-1,b-1,c-1)\Psi(a-2,b-2,c-1).
\end{split}
\end{equation}
\end{lem}

\begin{figure}\centering
\includegraphics[width=14cm]{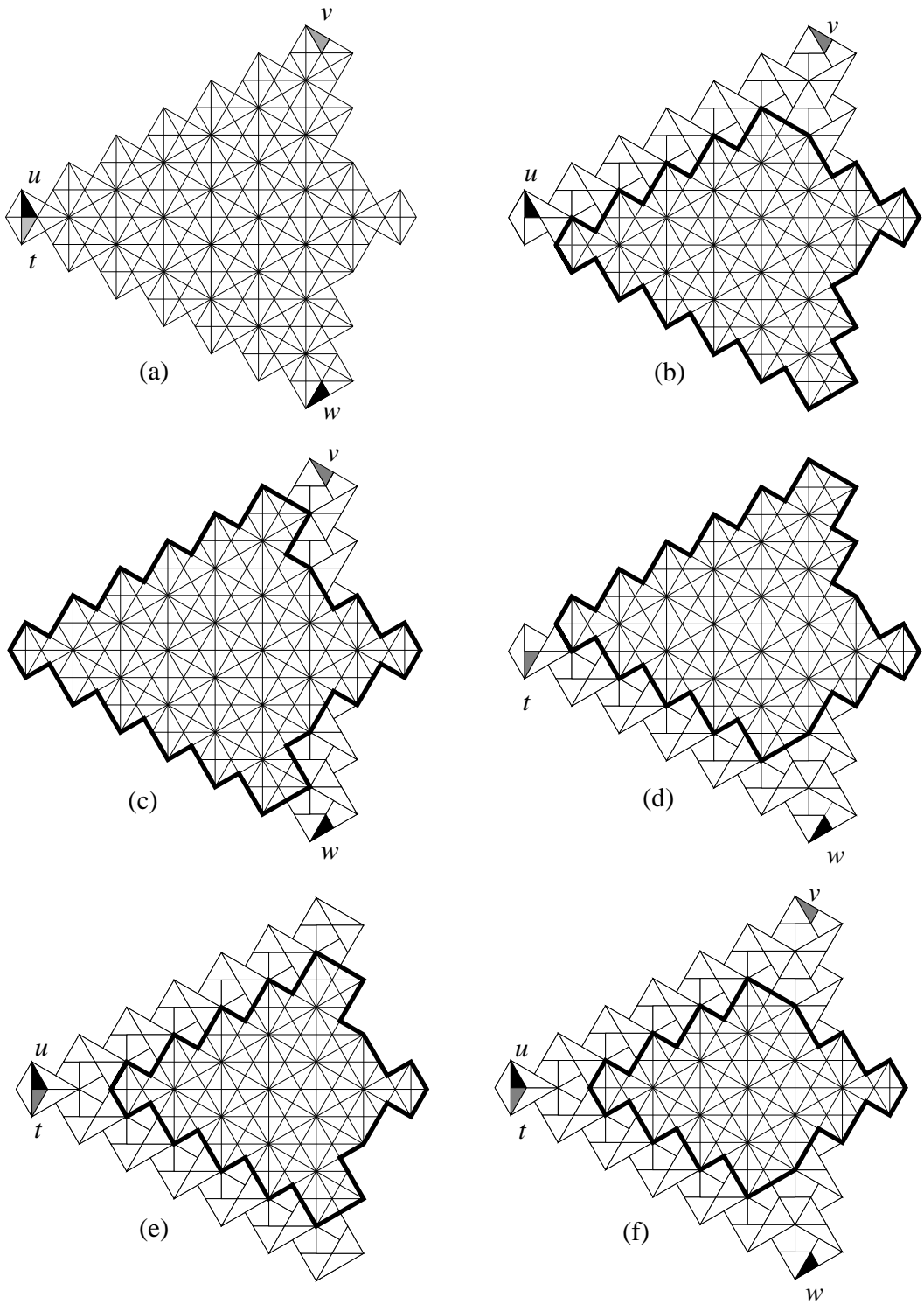}
\caption{Illustrating the proof of equality (\ref{eq1-1}) for the case of region $D_{8,8,3}(x,y)$. }
\label{dualcombine}
\end{figure}


\begin{figure}\centering
\includegraphics[width=6cm]{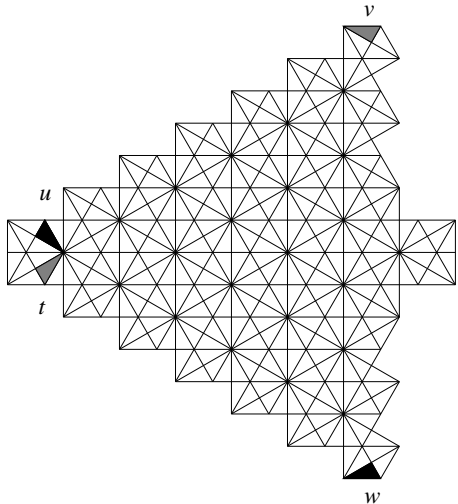}
\caption{How we apply Kuo condensation to the dual graph of $E_{8,8,3}(x,y)$. The black triangles correspond to the vertices $u$ and $w$; the shaded triangles correspond to the vertices $v$ and $t$.}
\label{conden2}
\end{figure}

\begin{proof} We assume that each obtuse triangle tile of $D_{a,b,c}$ and $E_{a,b,c}$ is weighted by $x$, each equilateral triangle tile is weighted by $y$, and kite tiles have weight 1.  To specify the weight assignment, we use the notations $D_{a,b,c}(x,y)$ and $E_{a,b,c}(x,y)$ for the corresponding weighted versions of $D_{a,b,c}$ and $E_{a,b,c}$. Thus, the functions $\Phi(a,b,c)$ and $\Psi(a,b,c)$ are exactly $\M(D_{a,b,c}(x,y))$ and $\M(E_{a,b,c}(x,y))$, respectively.

Apply Kuo's Condensation Theorem \ref{kuothm} to the dual graph $G$ of the weighted region $D_{a,b,c}(x,y)$ with the four vertices $u,v,w,t$ chosen as in Figure \ref{dualcombine}(a). More precisely, the black triangles on the west and south corners of region $D_{a,b,c}(x,y)$ correspond respectively to the vertices $u$ and $w$; the shaded triangles on the west and north corners correspond to the vertices $t$ and $v$.

 Consider the region corresponding to the graph $G-\{u,v\}$. The region has some tiles, which are forced to be in any tilings. By removing these edges, we get  region $D_{a-2,b-1,c}(x,y)$ (see the region restricted by the bold contour in Figure \ref{dualcombine}(b)) and obtain \[\M(G-\{u,v\})=W\cdot\M(D_{a-2,b-1,c}(x,y))=W\cdot\Phi(a-2,b-1,c),\] where $W$ is the product of weights of all forced tiles. Thus, by collecting the weights of the forced tiles, we get
\begin{equation}\label{con1a1}
\M(G-\{u,v\})=x^{a+f-3}y^{2a+2f-7}\Phi(a-2,b-1,c),
\end{equation}
where $f:=|2a+c-2b|$ as usual.

Similarly, we get four more identities, which are respectively illustrated in Figures \ref{dualcombine}(c)--(f):
\begin{equation}\label{con1a2}
\M(G-\{v,w\})=x^{c+f-2}y^{2c+2f-4}\Phi(a-1,b-1,c-1),
\end{equation}
\begin{equation}\label{con1a3}
\M(G-\{w,t\})=x^{b+c-3}y^{2b+2c-7}\Phi(a-1,b-2,c-2),
\end{equation}
 \begin{equation}\label{con1a4}
 \M(G-\{t,u\})=x^{a+b-3}y^{2a+2b-8}\Phi(a-2,b-2,c-1),
 \end{equation}
 and
\begin{equation}\label{con1a5}
\M(G-\{u,v,w,t\})=x^{a+b+c+f-7}y^{2a+2b+2c+2f-16}\Phi(a-3,b-3,c-2).
\end{equation}
Substituting the above five equalities (\ref{con1a1})--(\ref{con1a5}) into the identity  (\ref{kuoeq}) in Theorem \ref{kuothm}, we obtain (\ref{eq1-1}).

Similarly, we get (\ref{eq2-1}) by applying Theorem \ref{kuothm} to the dual graph of the weighted region $E_{a,b,c}(x,y)$ with the four vertices $u,v,w,t$ chosen as in Figure \ref{conden2}.
\end{proof}
One readily sees that  Lemma \ref{con1} implies Lemma 4.1 in \cite{CL} by specializing $x=y=1$.

\begin{figure}\centering
\includegraphics[width=14cm]{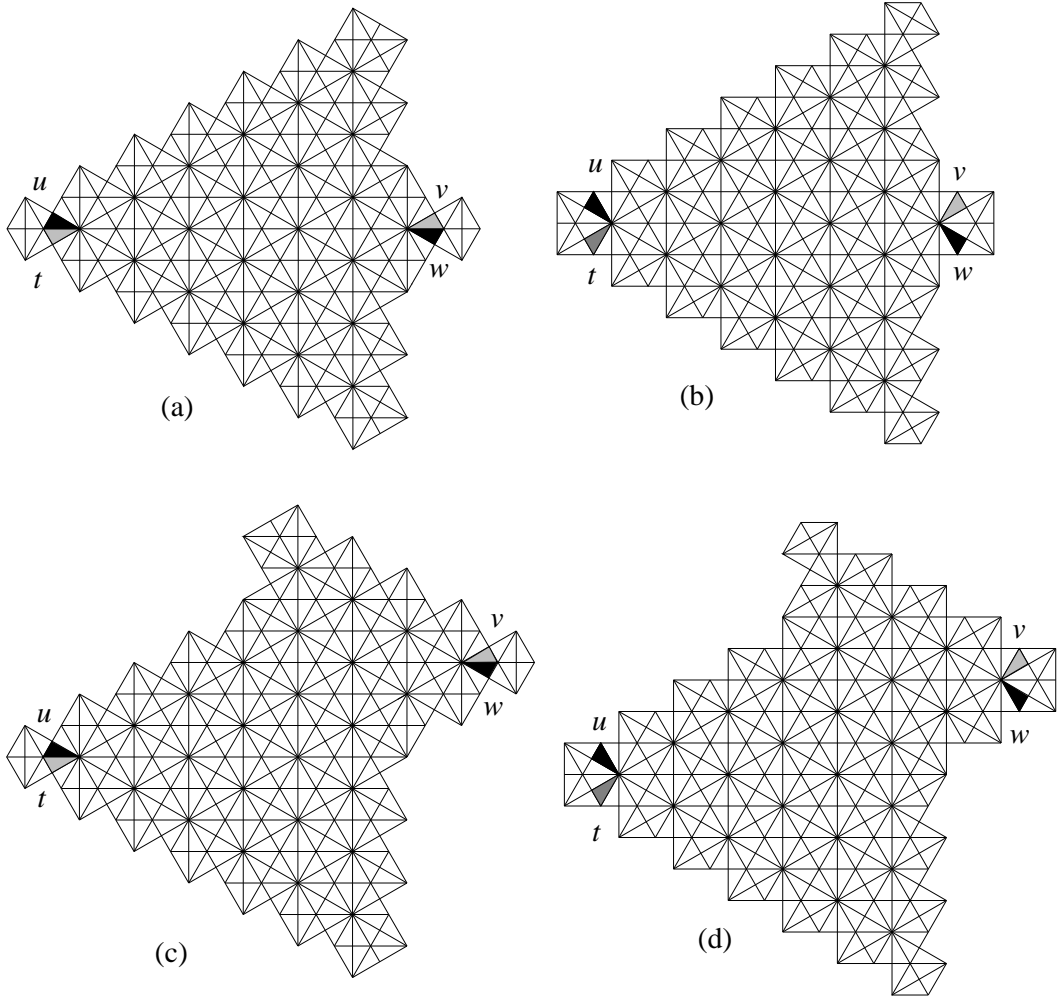}
\caption{How to apply Kuo condensation in Lemma \ref{con3} to the regions (a) $D_{8,8,3}(x,y)$, (b) $E_{8,8,3}(x,y)$, (c) $D_{5,8,4}(x,y)$ and (d) $E_{5,8,4}(x,y)$. }
\label{newcondense}
\end{figure}

\medskip

Applying also the Kuo's Theorem \ref{kuothm} to the dual graphs of $D_{a,b,c}(x,y)$ and $E_{a,b,c}(x,y)$ with the four vertices $u,u,w,t$ chosen as in Figure \ref{newcondense}, we get the following weighted version of Lemma 4.2 in \cite{CL}.
\begin{lem}\label{con3}    Let $a$, $b$ and $c$ be non-negative integers satisfying $a\geq 2$, $b\geq4$, $d:=2b-a-2c\geq2$, and $e:=3b-2a-2c\geq2$.

$(${\rm a}$)$. If $c\geq 1$, then $\Phi(a,b,c)$ and $\Psi(a,b,c)$ both satisfy the recurrence (\ref{WR2}).

$(${\rm b}$)$. If  $c=0$, then $\Phi(a,b,c)$ and $\Psi(a,b,c)$ both satisfy also the recurrence (\ref{WR3}).
%
 \end{lem}

\begin{figure}\centering
\includegraphics[width=14cm]{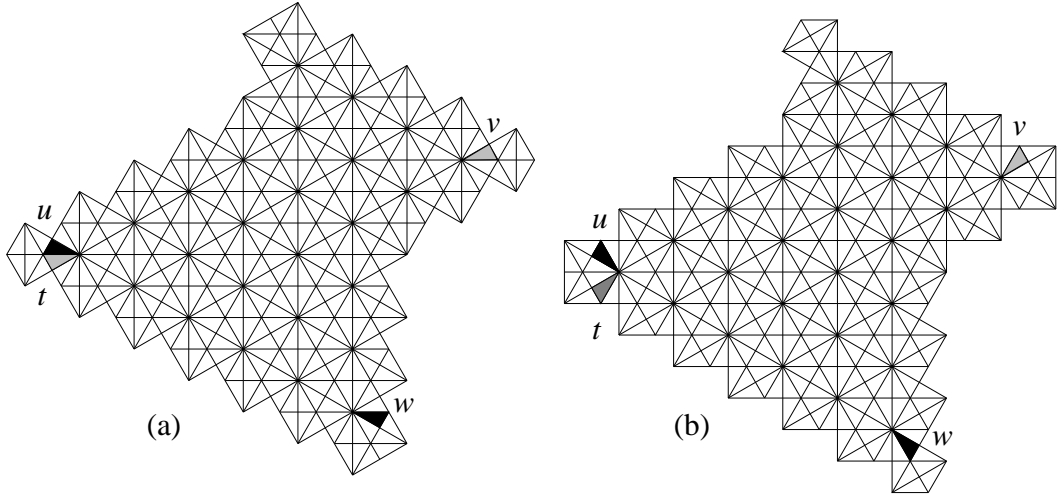}
\caption{How apply Kuo condensation in Lemma \ref{con4} to the regions (a) $D_{5,8,4}(x,y)$ and (b) $E_{5,8,4}(x,y)$. }
\label{condensation5n}
\end{figure}

A similar generalization of Lemma 4.3 in \cite{CL} is obtained by applying Kuo condensation to the dual graphs of $D_{a,b,c}(x,y)$ and $E_{a,b,c}(x,y)$ with the four vertices $u,v,w,t$ selected as in Figure \ref{condensation5n}.

 \begin{lem}\label{con4} Assume that $a,b,c$ are three non-negative integers satisfying $a\geq2$, $b\geq5$, $c\geq2$, $d:=2b-a-2c\geq0$, and $e:=3b-2a-2c\geq0$. Assume in addition that $a\leq c+d$.

 $(${\rm a}$)$. If $d\geq1$, then  $\Phi(a,b,c)$ and $\Psi(a,b,c)$  both satisfy (\ref{WR4}).
%

 $(${\rm b}$)$. If $d=0$, then the pair of functions $(\Phi(a,b,c),\Psi(a,b,c))$  satisfies the double-recurrence (\ref{WR5}).
\end{lem}

\medskip
We now denote by $\phi(a,b,c)$ and $\psi(a,b,c)$ the expressions on the right hand sides of the equalities (\ref{maineq1}) and (\ref{maineq2}) in Theorem \ref{newmain}, respectively.
It is routine to verify that the two functions $\phi(a,b,c)$ and $\psi(a,b,c)$ satisfy the same recurrences (\ref{WR1})--(\ref{WR5}) for any $a,b,c\in \mathbb{Z}$ (If one wants a detailed verification, we recommend the proofs of Lemmas 5.1--5.3 in \cite{CL}).  This yields an inductive proof of Theorem \ref{newmain} on the perimeter of the contour $\mathcal{C}(a,b,c)$.

The base cases are the situations when at least one of the following hold:
\begin{enumerate}
\item[(1)] Perimeter of the contour if at most 14.
\item[(2)] $b\leq 4$.
\item[(3)] $c+d=2b-a-c\leq 2$.
\end{enumerate}
Similar to the proof of Theorem 3.1 in \cite{CL}. These base cases can be verified by the help of the computer package \texttt{vaxmacs} written by David Wilson\footnote{This software is available at  \texttt{http://dbwilson.com/vaxmacs/}}, together with \texttt{Maple} by \texttt{Maplesoft}.

For induction step, we assume that the theorem is true for any regions $D_{a,b,c}(x,y)$ and $E_{a,b,c}(x,y)$ with perimeter greater than or equal to $16$ (it is easy to check that the perimeter is always even). We consider several cases ans show that the number of tilings of the regions and the expressions on the right hand sides of (2.10) and (2.11) satisfy the same recurrence: (R1), (R2), $\dots$ or (R5). However, all arguments here are essentially the same as that in the induction step of the proof of  Theorem 3.1 in \cite{CL} (with the recurrences are replaced by their corresponding weighted versions). Therefore, we can finish our proof of Theorem \ref{newmain} here.

\section{Proof of Theorem \ref{weighthd}}

Before presenting the proof of Theorem \ref{weighthd}, we quote the following lemma.

\begin{lem}[Graph-Splitting Lemma; Lemma 3.6(a) in \cite{Tri1}]\label{GS}
Let $G$ be a bipartite graph, and let $V_1$ and $V_2$ be the two vertex classes.

Assume that an induced subgraph $H$ of $G$ satisfies following two conditions:
\begin{enumerate}
\item[(i)]  There are no edges of $G$ connecting a vertex in $V(H)\cap V_1$ and a vertex in $V(G-H)$.

\item[(ii)] $|V(H)\cap V_1|=|V(H)\cap V_2|$.
\end{enumerate}
Then
\begin{equation}
\M(G)=\M(H)\, \M(G-H).
\end{equation}
\end{lem}

\medskip

\begin{proof}[Proof of Theorem \ref{weighthd}]
To specify the weight assignment on tiles of the hexagonal dungeon $HD_{a,2a,b}$, we use the notation $HD_{a,2a,b}(x,y,z)$ for its weighted version with obtuse triangle tiles, equilateral triangle tiles, and kite tiles are weighted by $x,y,z$, respectively. It means that $F(x,y,z)$ is now $\M(HD_{a,2a,b}(x,y,z))$.  We now divide the weight of each tile in the region $HD_{a,2a,b}(x,y,z)$ by $z$. We get the new weighted region $HD_{a,2a,b}\left(\frac{x}{z},\frac{y}{z},1\right)$ and obtain
\begin{equation}\label{weighteq2}
\M(HD_{a,2a,b}(x,y,z))=2^{6(3ab+2a^2)-3(3a+b)}\M\left(HD_{a,2a,b}\left(\frac{x}{z},\frac{y}{z},1\right)\right),
\end{equation}
where $6(3ab+2a^2)-3(3a+b)$ is the total number of tiles in the region $HD_{a,2a,b}$. Thus, the theorem can be reduced to the case when $z=1$.

\begin{figure}\centering
\includegraphics[width=10cm]{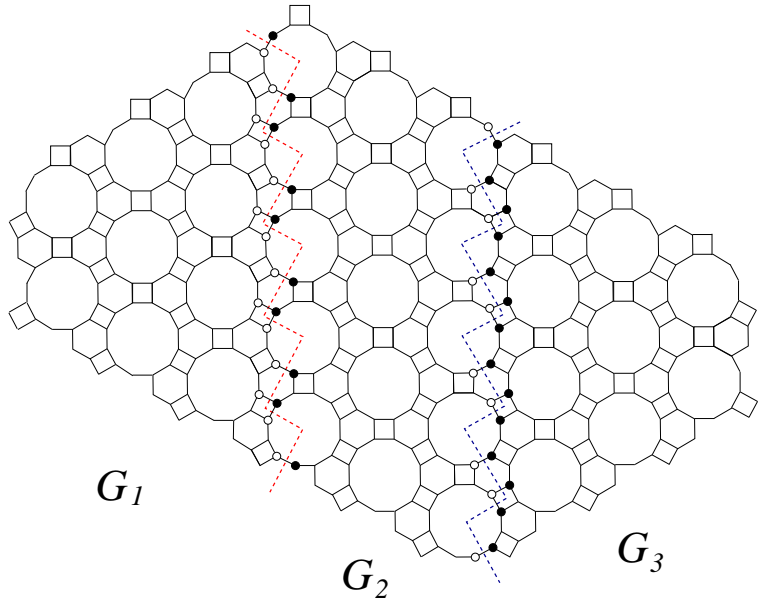}
\caption{Splitting the dual graph of the hexagonal dungeon $HD_{2,4,6}$ into three smaller graphs.}
\label{newsplit}
\end{figure}

\begin{figure}\centering
\includegraphics[width=10cm]{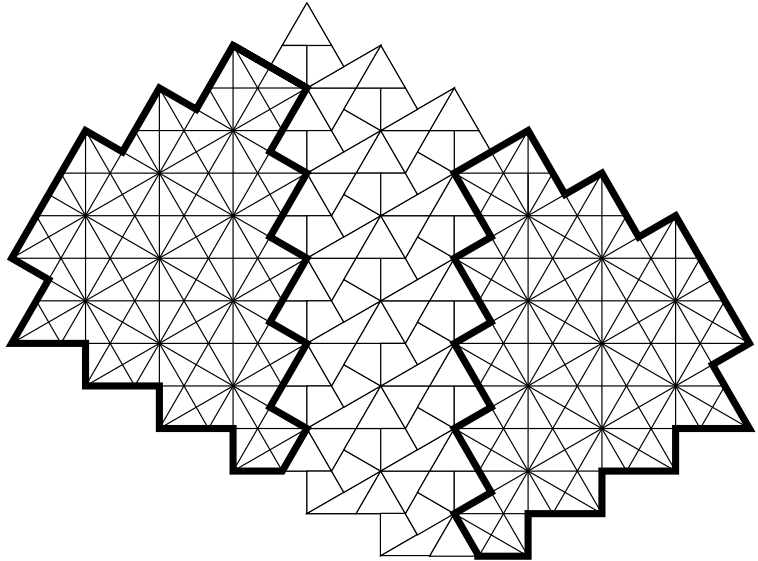}
\caption{The dual picture of Figure \ref{newsplit}.}
\label{newsplit2}
\end{figure}

We now prove the theorem in the case $z=1$.
Apply the two zigzag cuts  to the dual graph of the weighted region $HD_{a,2a,b}(x,y,1)$ as in  Figure \ref{newsplit} to divide it into three disjoint subgraphs $G_{1}$, $G_2$, and $G_3$ (in order from left to right). These subgraphs satisfy the conditions in Graph-splitting Lemma \ref{GS}, so we obtained
\begin{equation}\label{weighteq}
\M\left(HD_{a,2a,b}(x,y,1)\right)=\M\left(G_1\right)\M\left(G_2\right)\M\left(G_3\right).
\end{equation}

The graph $G_1$ and $G_3$ are both isomorphic to the dual graph of the region $D_{2a,3a,2a}(x,y)$, where each obtuse triangle tile is weighted by $x$, and each equilateral tile is weighted  by $y$ (see the regions with bold boundaries in Figure \ref{newsplit2}). Thus, $\M(G_1)=\M(G_3)=\Phi(2a,3a,2a)$. Moreover, the graph $G_2$ corresponds to the middle region in Figure \ref{newsplit2}, which has only one tiling (all tiles are forced with the pattern as in Figure \ref{newsplit2}). By counting the number of each type of tiles in the unique tiling of the middle region in Figure \ref{newsplit2}, one obtains that the weight of the perfect matching in $G_2$ is
 $x^{3a(b-2a)}y^{3a(b-2a+1)+(3a-1)(b-2a)}$. By (\ref{weighteq}), we get
\begin{equation}
\M\left(HD_{a,2a,b}(x,y,1)\right)=x^{3a(b-2a)}y^{3a(b-2a+1)+(3a-1)(b-2a)}\Phi(2a,3a,2a)^2.
\end{equation}

By definition, we obtain
$q(2a,3a,2b)=\lfloor\frac{a^2}{4}\rfloor$,
$g(2a,3a,2a)=a^2$, and
$h(2a,3a,2a)=1$ if $a$ is even, and  $y$ if $a$ is odd (we are assuming $z=1$).
Thus, we get
\begin{equation}
\M\left(HD_{a,2a,b}(x,y,1)\right)=\gamma|_{z=1}(a)x^{3ab-2a^2+3\lfloor\frac{a^2}{2}\rfloor} y^{6ab-a^2+3a-b-2\lfloor\frac{a^2}{2}\rfloor}P(x,y,1)^{2a^2}Q(x,y,1)^{\lfloor\frac{a^2}{2}\rfloor},
\end{equation}
where $\gamma|_{z=1}(a)$ is the restriction  of $\gamma(a)$ at $z=1$. This implies the theorem for $z=1$, and (\ref{weighteq2}) in turn deduces the theorem in the general case.
\end{proof}

\section{New hexagonal dungeons}


\begin{figure}\centering
\includegraphics[width=5cm]{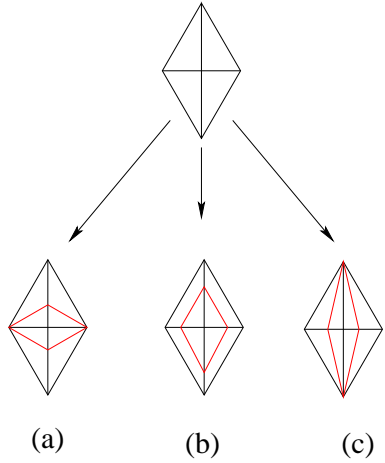}
\caption{Three replacement rules for the rhombi in the $G_2$-lattice.}
\label{replacerule}
\end{figure}

In this section, we investigate several new hexagonal dungeons.


The $G_2$-lattice can be partitioned into rhombi with two diagonals drawn in (see the top picture in Figure \ref{replacerule}).  Apply the replacing rule (a) in Figure \ref{replacerule} to all the rhombi, and denote by $G_{2}^{(1)}$ the resulting lattice. We have a variant of the hexagonal dungeon $HD_{a,2a,b}$ as in Figure \ref{hexagondnew1} (see the region restricted by the dark bold jagged contour). Denote by $HD^{(1)}_{a,2a,b}$ the new hexagonal dungeon.


%

\begin{figure}\centering
\includegraphics[width=10cm]{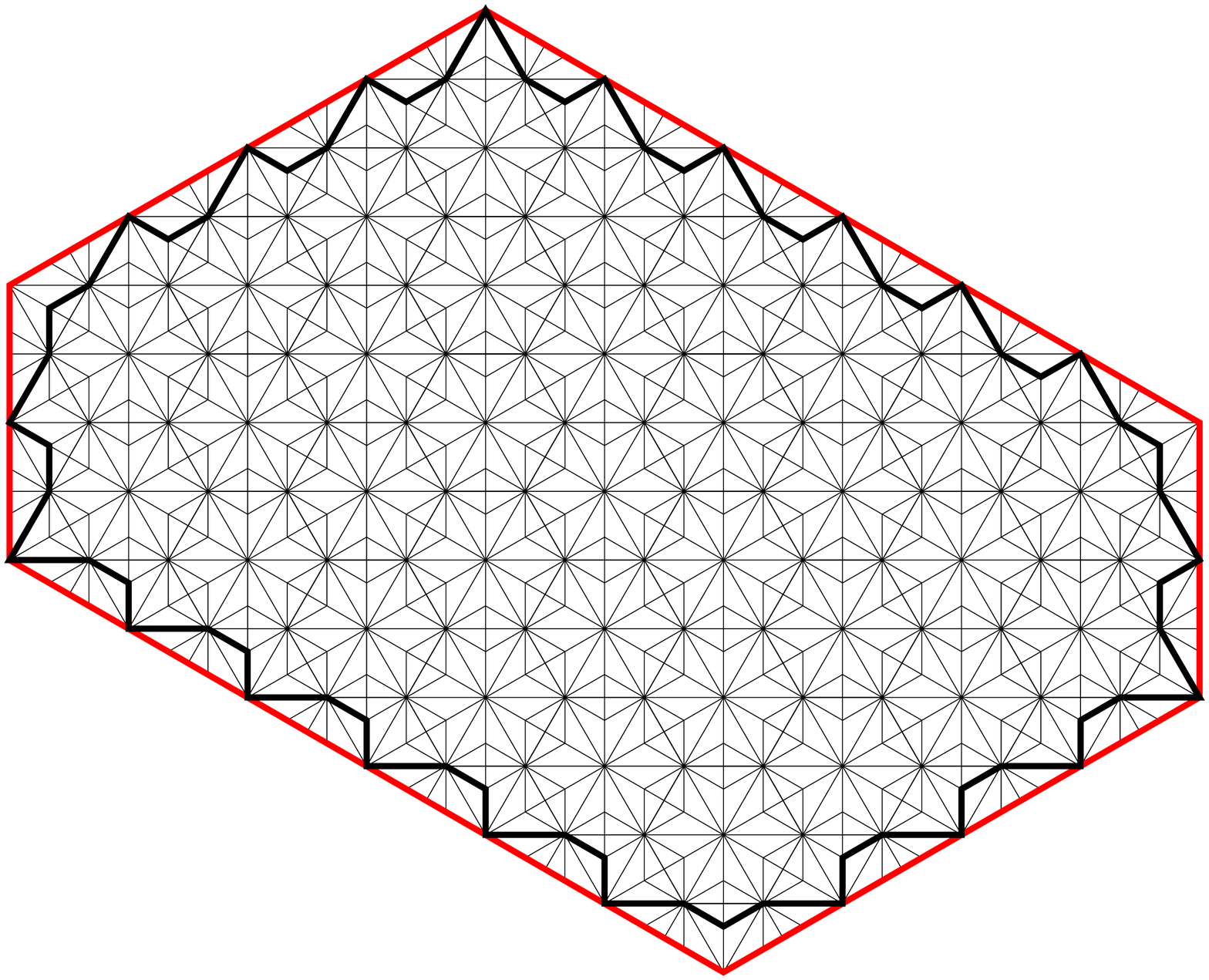}
\caption{The variant $HD^{(1)}_{2,4,6}$ of the hexagonal dungeon $HD_{2,4,6}$ on the $G_2^{(1)}$-lattice.}
\label{hexagondnew1}
\end{figure}

\begin{cor}\label{v2}
%
 Assume that $a$ and $b$ are two positive integers, so that $b\geq 2a$. Then
\begin{equation}
\M\left(HD^{(1)}_{a,2a,b}\right)
=\gamma^{(1)}(a) 2^{9a^2-6a-3\lfloor\frac{a^2}{2}\rfloor}3^{3a^2+2\lfloor\frac{a^2}{2}\rfloor+6ab+3a-b+\gamma}17^{2a^2},
\end{equation}
where $\gamma^{(1)}(a)$ is 1 if $a$ is even, and $9/4$ if $a$ is odd.
\end{cor}

\begin{figure}\centering
\includegraphics[width=10cm]{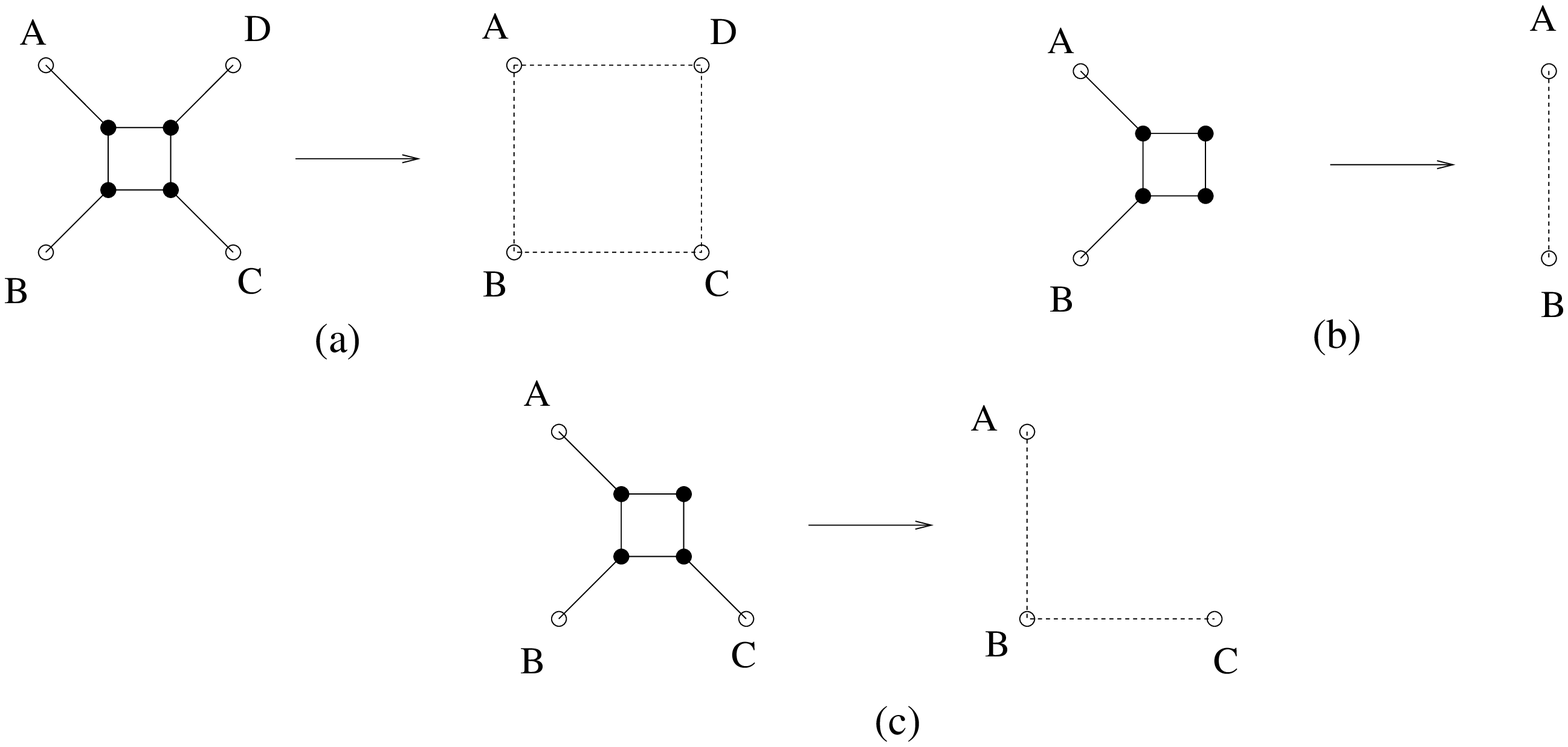}
\caption{Urban renewal trick.}
\label{urban}
\end{figure}

%

Next, we quote a well-known subgraph replacement trick called \textit{urban renewal}, which was first discovered by Kuperberg.

\begin{lem}[Urban renewal]\label{spider}
 Let $G$ be a weighted graph. Assume that $G$ has a subgraph $K$ as one of the graphs on the left column in Figure \ref{urban}, where only white vertices can have neighbors outside $K$, and where all edges have weight $1$. Let $G'$ be the weighted graph obtained from $G$ by replacing $K$ by its corresponding graph $K'$ on the right as in Figure \ref{urban}, where all dotted edges have weight $\frac{1}{2}$. Then we always have $\M(G)=2\M(G')$.
 \end{lem}

We have a small observation as follows. If there are several parallel edges connecting the same vertices $u$ and $v$ in a graph $G$, then the number of perfect matchings of $G$ doesn't change if the  we can replace these parallel edges by a new single edge connecting $u$ and $v$ with the weight equal to the sum of weights of the original edges (see Figure \ref{edgereplace}).

\begin{figure}\centering
\includegraphics[width=10cm]{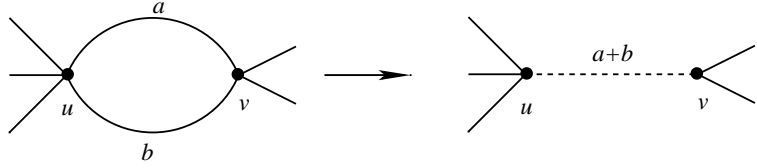}
\caption{Replacing parallel edges by a single edge.}
\label{edgereplace}
\end{figure}

\begin{figure}\centering
\includegraphics[width=12cm]{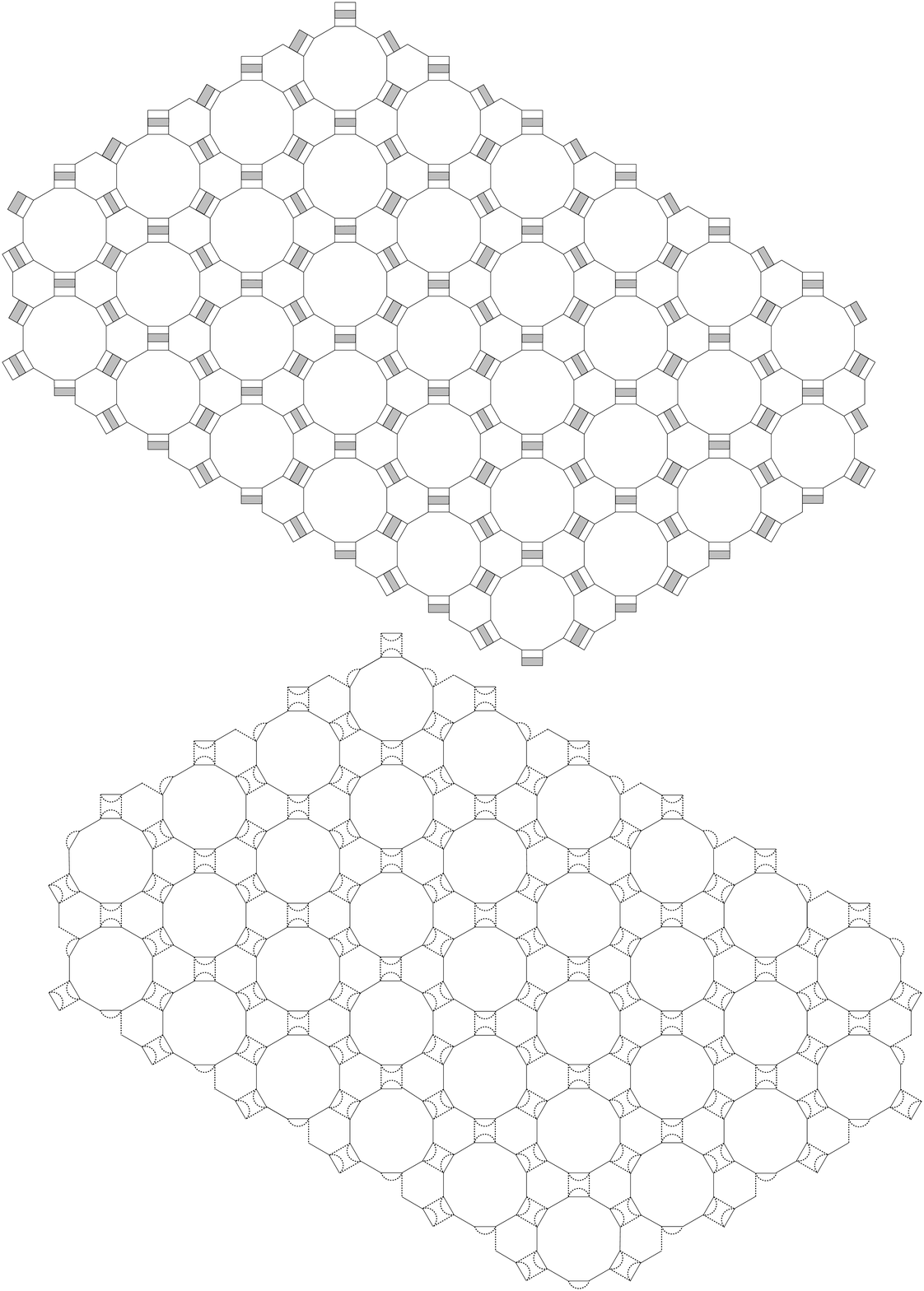}
\caption{The dual graph of $HD^{(1)}_{2,4,6}$ (upper) and the graph obtained from it by applying urban renewal (lower). The dotted edges have weight $1/2$.}
\label{dualdungeon1}
\end{figure}

\begin{proof}[ Proof of Corollary \ref{v2}]
Consider the dual graph $G$ of $HD^{(1)}_{a,2a,b}$ (see the upper picture in Figure \ref{dualdungeon1}). Apply suitable replacements in Lemma \ref{spider} around all shaded rectangles in $G$. These replacements create several pairs of parallel edges (consisting of an edge of weight $1$ and an edge of weight $1/2$) in the resulting graph (see the lower picture in Figure \ref{dualdungeon1}).  Next, we replace each pair of parallel edges by a new single edge of weight $3/2$. This way the dual graph $G$ of  $HD^{(1)}_{a,2a,b}$ is transformed into the dual graph of the weighted hexagonal dungeon  $HD_{a,2a,b}(1/2,3/2,1)$. Since there are $C:=(6a-1)(b-2a)+3a(b-2a+1)+4a(2a-1)+4a(4a-1)$ shaded rectangles in $G$, we have
\begin{equation}
\M\left(HD^{(1)}_{a,2a,b}\right)=\M(G)=2^{C} \M\left(HD_{a,2a,b}(1/2,3/2,1)\right).
\end{equation}
Then the corollary follows from Theorem \ref{weighthd}.
\end{proof}


Next, we consider the application of the replacement rule (b) in Figure \ref{replacerule} to all rhombi of $G_2$-lattice, and denote by $G_{2}^{(2)}$ the resulting lattice. On the $G_2^{(2)}$-lattice, we have a new version of the hexagonal dungeons illustrated by the region restricted by the dark bold contour in Figure \ref{hexagondnew2}.

%
%
%

\begin{figure}\centering
\includegraphics[width=10cm]{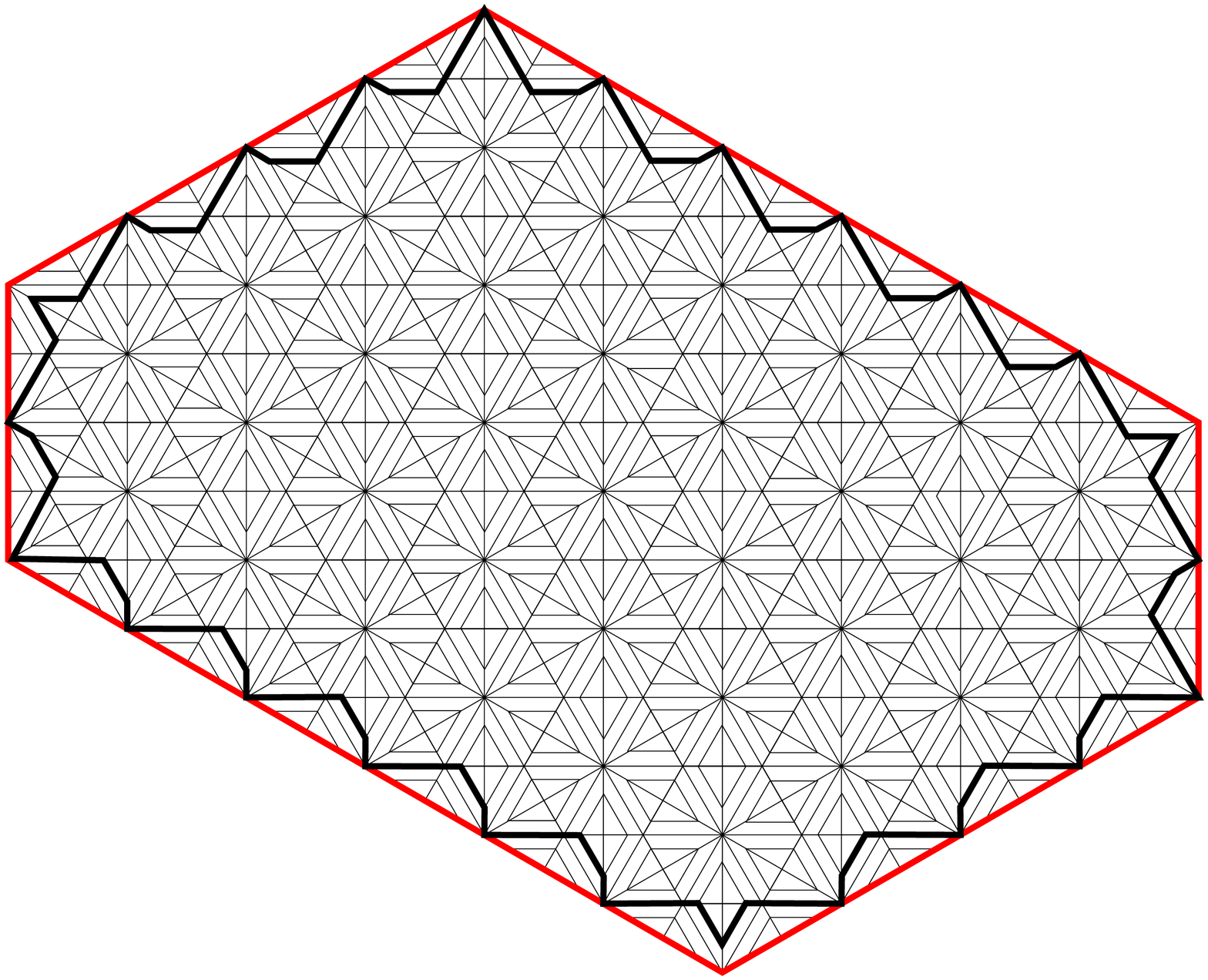}
\caption{The variant $HD^{(2)}_{2,4,6}$ of the hexagonal dungeon $HD_{2,4,6}$ on the $G_2^{(2)}$-lattice.}
\label{hexagondnew2}
\end{figure}

\begin{figure}\centering
\includegraphics[width=13cm]{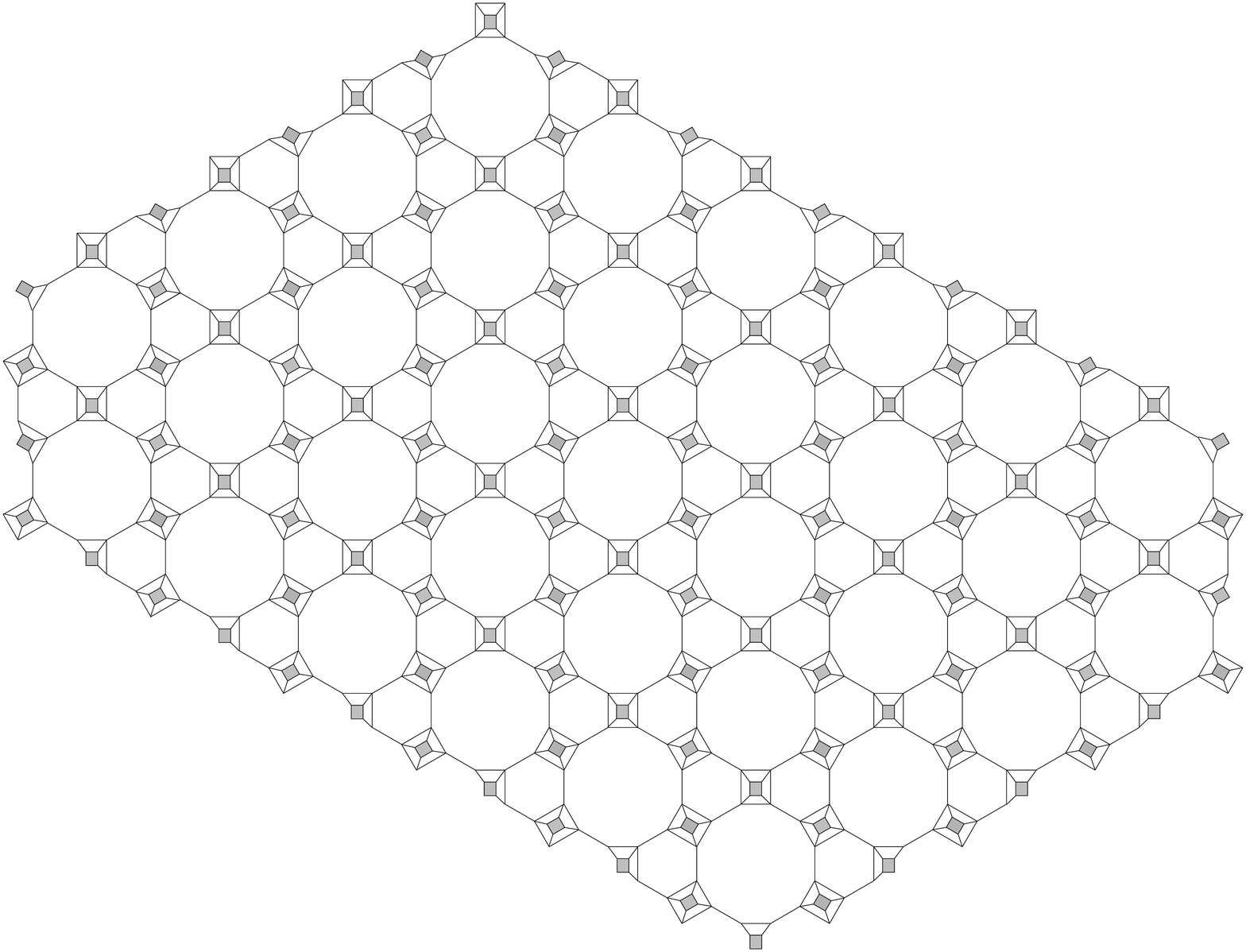}
\caption{The dual graph of $HD^{(2)}_{2,4,6}$.}
\label{dualdungeon3}
\end{figure}

\begin{cor}\label{v3}
%
 Assume that $a$ and $b$ are two positive integers, so that $b\geq 2a$. Then
\begin{equation}
\M\left(HD^{(2)}_{a,2a,b}\right)
=\gamma^{(1)}(a)2^{9a^2-6a-\lfloor\frac{a^2}{2}\rfloor}3^{9ab-3a^2+\lfloor\frac{a^2}{2}\rfloor+3a-b}5^{2a^2}13^{4a^2+\lfloor\frac{a^2}{2}\rfloor}109^{\lfloor\frac{a^2}{2}\rfloor},
\end{equation}
where $\gamma^{(1)}(a)$ is defined as in Corollary \ref{v2}.
\end{cor}

\begin{proof}
Similar to the previous corollary, the dual graph of $HD^{(2)}_{a,2a,b}$  can be transformed into the  dual graph of $HD_{a,2a,b}(3/2,3/2,1)$ by applying urban renewal around all $C$ shaded squares ($C$ is defined as in Corollary \ref{v2}), and replacing each pair of parallel edges by a single edge with weight $3/2$ (see Figure \ref{dualdungeon3}). We have
\begin{equation}
\M\left(HD^{(2)}_{a,2a,b}\right)=2^{C} \M\left(HD_{a,2a,b}(3/2,3/2,1)\right),
\end{equation}
and the corollary follows again from Theorem \ref{weighthd}.
\end{proof}

%

\begin{figure}\centering
\includegraphics[width=10cm]{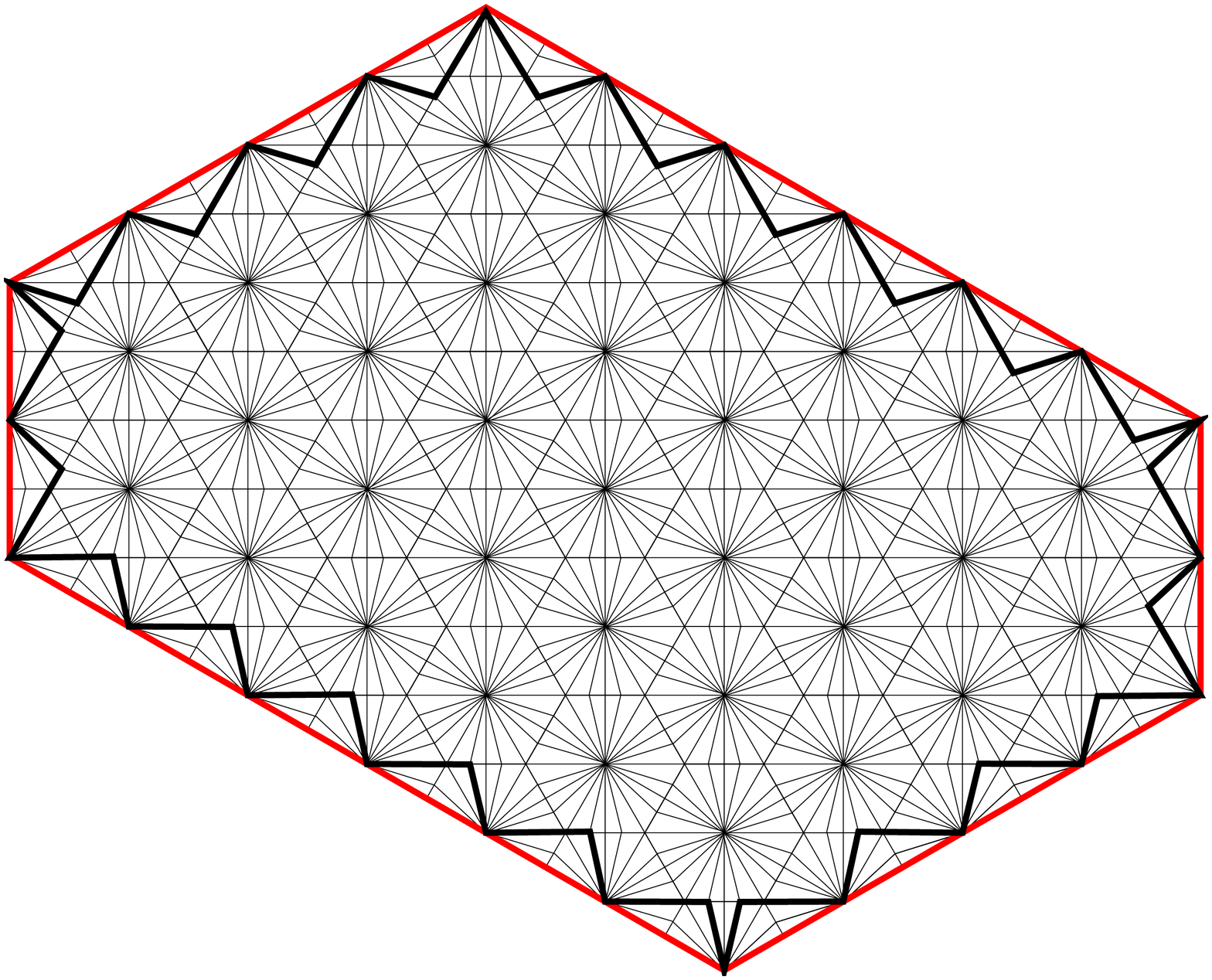}
\caption{The variant $HD^{(3)}_{2,4,6}$ of the hexagonal dungeon $HD_{2,4,6}$ on the $G_2^{(3)}$-lattice.}
\label{hexagondnew3}
\end{figure}

\begin{figure}\centering
\includegraphics[width=13cm]{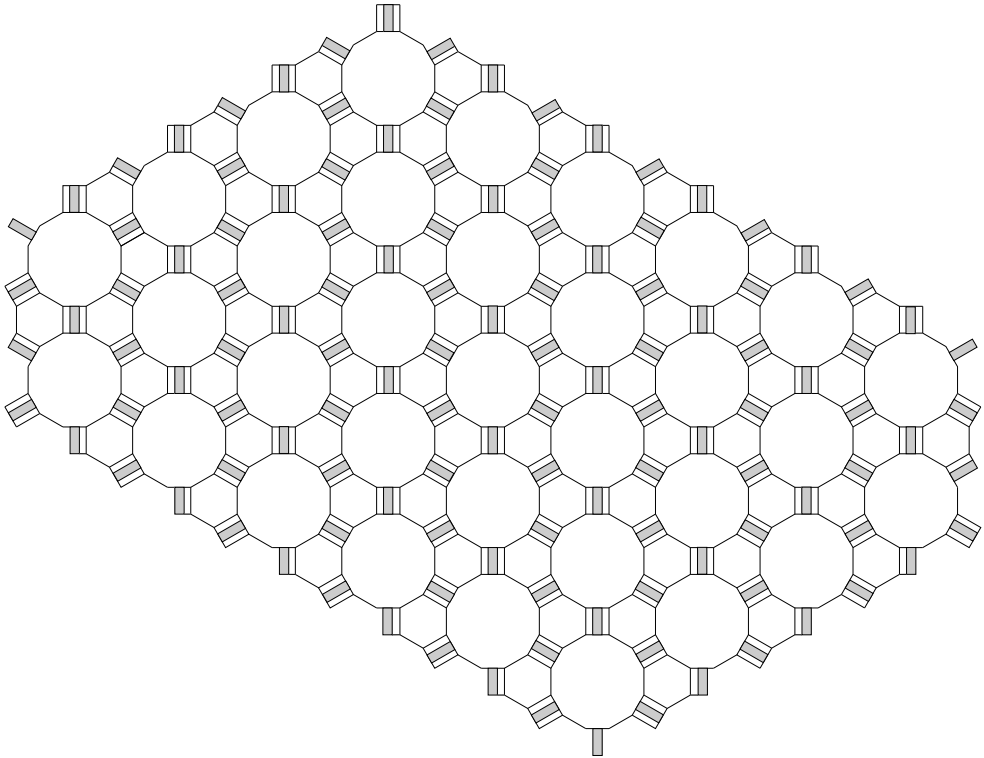}
\caption{The dual graph of $HD^{(3)}_{2,4,6}$.}
\label{dualdungeon2}
\end{figure}

Finally, application of the replacement (c)  in Figure \ref{replacerule}  gives us a new lattice $G_2^{(3)}$ as in Figure \ref{hexagondnew3}. On the new lattice, we get the variant $HD^{(3)}_{a,2a,b}$ of region  $H_{a,2a,b}$ as the region restricted by dark bold contour in Figure \ref{hexagondnew3}. The tilings of the above new regions are enumerated by powers of $2$, $3$, $5$, and $193$ as follows.

\begin{cor}\label{v1}
 Assume that $a$ and $b$ are two positive integers, so that $b\geq 2a$. Then
\begin{equation}\label{coreq3}
\M\left(HD^{(3)}_{a,2a,b}\right)
=\gamma^{(3)}(a)2^{9a^2-6a-2\lfloor\frac{a^2}{2}\rfloor}3^{3ab-2a^2+5\lfloor\frac{a^2}{2}\rfloor}5^{\lfloor\frac{a^2}{2}\rfloor}193^{2a^2},
\end{equation}
where $\gamma^{(3)}(a)$ is 1 if $a$ is even, and $1/4$ if $a$ is odd.
\end{cor}

\begin{proof}
This corollary can be obtained similarly to Corollaries \ref{v2} and \ref{v3}. We apply urban renewal at all $C$ shaded rectangles, where $C$ is defined as in Corollary \ref{v2} (see Figure \ref{dualdungeon2}). Then we replace each pair of parallel edges by a single edge of weight $3/2$. This way, we transform the dual graph of $HD^{(3)}_{a,2a,b}$ into the dual graph of  $HD_{a,2a,b}(3/2,1/2,1)$.  Again, Theorem \ref{weighthd} deduces (\ref{coreq3}).
\end{proof}

\begin{figure}\centering
\includegraphics[width=8cm]{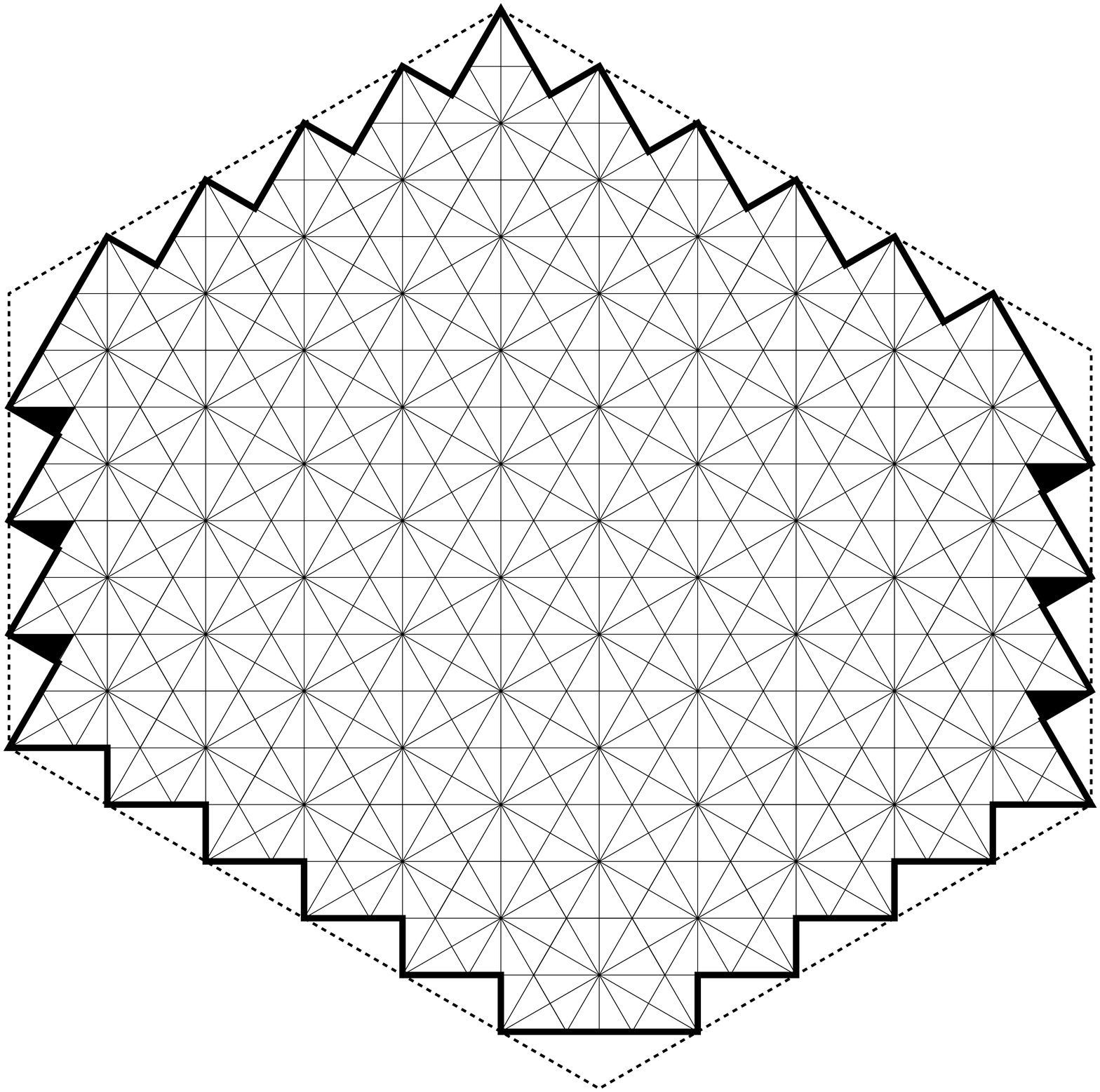}
\caption{The hexagonal dungeon with defects $HH_{4,6}$. The black triangles indicate the fundamental regions removed.}
\label{HDhole}
\end{figure}

\section{An open question on hexagonal dungeons with defects.}

We conclude this paper by considering a hexagonal dungeon where some fundamental regions have been removed as follows. Similar to the original hexagonal dungeon, we consider a similar region that is restricted by a jagged boundary running along the hexagonal contour of side-lengths $a,a+1,b,a,a+1,b$ (in cyclic order, starting from the west side). Next, we remove $a-1$ triangles from the region from each of the west and  east sides, and denote by $HH_{a,b}$ the resulting region. Figure \ref{HDhole} shows the region $HH_{4,6}$; the black triangles indicate the fundamental regions removed.

We consider the  tiling generating function of $HH_{a,b}$ as
\begin{equation}
K(x,y,z):=\sum_{T\in\mathcal{M}(HH_{a,b})}x^{m}y^{n}z^{l},
\end{equation}
 where $m$, $n$, $l$ are respectively the numbers of obtuse triangle tiles, equilateral tiles, and kite tiles in the tiling $T$ as usual. It seems that $K(x,y,z)$ is given by a simple product similar to $F(x,y,z)$ in Theorem \ref{weighthd}.

%
%
%



\begin{conj}
Assume that $a$ and $b$ are two positive integers, so that $b\geq a+1$. Then the generating function $K(x,y,z)$ always has form
\[x^{X}y^{Y}z^{Z}P(x,y,z)^{L},\]
where $X,Y,Z,L$ depend only on $a$ and $b$.
\end{conj}
Since several triangles have been removed along the boundary of the region, our method in the paper seems does not work for $HH_{a,b}$. 

Finally, we notice that ones can create new versions of the region $HH_{a,b}$ on the lattices $G_2^{(1)}$, $G_2^{(2)}$ and $G_2^{(3)}$ as we did for the hexagon dungeons in the previous section.

\end{document}